\documentclass[12pt]{amsart}

\usepackage{amsmath,amsthm,amssymb}

\usepackage{fancyvrb }

\usepackage[margin=1in]{geometry}

\usepackage[hidelinks]{hyperref}

\usepackage{verbatim}

\usepackage{cases}

\usepackage{color}

\usepackage{todonotes}

\newtheorem{theorem}{Theorem}[section]

\newtheorem{proposition}[theorem]{Proposition}

\newtheorem{corollary}[theorem]{Corollary}

\newtheorem{lemma}[theorem]{Lemma}
\input xy
\xyoption{all}

\theoremstyle{definition}

\newtheorem{remark}[theorem]{Remark}

\DeclareMathOperator{\rank}{rank}

\DeclareMathOperator{\chr}{char}

\DeclareMathOperator{\Ext}{Ext}

\DeclareMathOperator{\opH}{H}
\DeclareMathOperator{\Hom}{Hom}

\newcommand{\fraku}{\mathfrak{u}}

\newcommand{\la}{\lambda}

\begin{document}

\title[Linear bounds for cohomology of algebraic groups]{Linear bounds for cohomology of algebraic groups}

\begin{abstract}
In this note, we establish an explicit upper bound for the dimension of the rational cohomology for a simple algebraic group over an algebraically closed field of prime characteristic.
\end{abstract}

\date{\today}

\author{\sc Christopher P. Bendel}
\address
{Department of Mathematics, Statistics and Computer Science\\
University of
Wisconsin-Stout Polytechnic\\
Menomonie\\ WI~54751, USA}
\thanks{Research of each author was supported in part by an AMS-Simons Research Enhancement Grant for PUI Faculty}
\email{bendelc@uwstout.edu}

\author{\sc Nham Ngo}
\address
{Department of Mathematics\\
University of North Georgia \\
Oakwood\\ GA~30566, USA}
\email{Nham.Ngo@ung.edu}

\maketitle

\section{Introduction}

This note builds on a long history of investigations into the dimension of the cohomology of a finite or algebraic group over an algebraically closed field $k$ of prime characteristic $p$.  Let $G$ be a finite group, $kG$ be its group algebra, and $M$ be a finite-dimensional $kG$-module.  A longstanding representation theory challenge has been to ascertain the dimension of a cohomology group $\opH^n(G,M)$ for a nonnegative integer $n$.  In particular, in the 1980s, Guralnick conjectured the existence of a universal bound on $\opH^1(G,M)$ for an absolutely irreducible  $kG$-module $M$, with further speculation that such a bound should be quite small.  That initial question later spawned investigations for higher degrees and more general modules, as well as to the setting of reductive (and related) algebraic groups.   In the algebraic group context, the early 2000s saw a number of results in this area as well as connections back to finite groups of Lie type.  A key event in this story was a workshop held at the American Institute of Mathematics (AIM) in the summer of 2012 out of which came strong evidence that Guralnick’s initial conjecture of a universal bound was unlikely to exist and signaled the need for broader approaches to the cohomological bounding question. For a more thorough summary of the history of the subject and highlights of that AIM meeting, the reader is directed to \cite{Ben}.   

In this work, we consider the question of bounding the dimension of $\opH^n(G,M)$ for a simple algebraic group (scheme) $G$ over $k$ and finite-dimensional rational $G$-module $M$.  Parshall and Scott \cite{PS} showed that for any {\em irreducible} $G$-module $M$, there exists a number $C(\Phi,n)$ depending only on the root system $\Phi$ and the cohomological degree $n$ such that \[ \dim\opH^n(G,M)\le C(\Phi,n).  \] The impressive work of Parshall and Scott uses deep representation-theoretic ideas and connections with quantum groups (including Kazhdan-Lusztig polynomials and the validity of the Luzstig Character Formula) to show the existence of such a bound more generally for extensions between simple $G$-modules, although no explicit bound is given.   In the case $n = 1$, Cline, Parshall, and Scott \cite{CPS09} (and reproved in \cite{PS}) did give an explicit bound depending on $p$ and Kostant's partition function.  Again in degree 1, Parker and Stewart \cite{PaSt} gave the first explict bound on the dimension of $\opH^1(G,M)$ for irreducible $M$ that depends only on the Coxter number of the group.

For an arbitrary (finite-dimensional) module $M$, a potentially more tractable task is to find a bound in terms of the dimension of $M$, paralleling results for finite groups from work of Guralnick et. al. \cite{GKKL}.  For example, in \cite{Ben}, it was observed that $\dim\opH^1(G,M) \leq \frac12\dim M$ and $\dim\opH^2(G,M) \leq \dim M$ (for any prime $p$).    In general, the “linear boundedness” question asks whether there exists a constant $c(n)$ such that $\dim\opH^n(G,M) \leq c(n)\cdot\dim M$ for any $M$, preferably with $c(n)$ independent of the prime $p$.   While the aforementioned degree 1 and 2 results are independent of the root system, the results of Parshall and Scott (and other work in that vein) suggest that $c(n)$ may need to depend on the root system for larger $n$.  Indeed, it follows from \cite{PS} that there exists a number $c(\Phi,n)$ such that 
\[\dim\opH^n(G,M)\le c(\Phi,n)\dim(M).\]

The goal of this work is to provide an explicit value for such a $c(\Phi,n)$ using much more elementary techniques (see Corollary \ref{C:main}).  Our bound is independent of the prime $p$, although a better bound can be obtained for $p > 3$.  The approach taken here is to consider a Borel subgroup $B$ of $G$ and a one-dimensional $B$-module associated to a dominant weight $\la$.  We identify a number $c(\Phi,n)$ such that
\[\dim\opH^n(B,\la)\le c(\Phi,n)\]
for any dominant weight $\la$ (see Theorems \ref{T:p>2} and \ref{T:p=2}). The proof uses combinatorial methods based on a well-known spectral sequence.  

Beyond the inherent interest in cohomology for algebraic groups, relationships between the cohomology of $G$ and the associated finite group $G({\mathbb F}_{p^r})$ of Lie type allow one to use bounds on $G$-cohomology to obtain bounds on $G({\mathbb F}_{p^r})$-cohomology. See for example \cite{CPS09}, \cite{PS}, \cite{Ben}, \cite{PaSt}, and \cite{BNPPSS}.


\section{Some preliminaries}\label{notation}

In the remainder, $G$ is a simple algebraic group defined over an algebraically closed field $k$ with $\chr(k)=p>0$. We use standard notation as in Jantzen's book \cite{Jan:2003}. For the convenience of the reader, we list necessary notation and conventions as follows.

\begin{itemize}
\item $B$: a Borel subgroup of $G$ corresponding to negative roots
\item $T$: a maximal torus in $B$
\item $U$: the unipotent radical of $B$ 
\item $X(T)$: the weight lattice of $T$
\item $X^{+} \subset X(T)$: the dominant weights
\item For a given positive integer $r$, let $F_r:G\to G$ be the $r$-th Frobenius morphism, see for example \cite[I.9]{Jan:2003}. The scheme-theoretic kernel $G_r=\ker(F_r)$ is called the $r$-th Frobenius kernel of $G$. 
\item Given a closed subgroup $H$ of $G$, denote $H_r$ the scheme-theoretic intersection $H\cap G_r$. 
\item $h_G$: Coxeter number of $G$
\item $\ell=\rank(G)$
\item $\Phi^+$: the set of positive roots. Similarly, $\Phi^-$ is the set of negative roots.
\item $\Phi=\Phi^+\cup\Phi^-$
\item $\Pi=\{\alpha_1,\ldots,\alpha_{\ell} \}$: the set of simple roots
\item $\mathbb{N} = \{0, 1, 2, 3, \dots\}$: the set of natural numbers
\item Each weight $\lambda\in\mathbb Z\Phi$ can be written
\[\lambda=c_1\alpha_1+\cdots+c_{\ell}\alpha_{\ell}\]
with $c_i\in\mathbb Z$. 
\item $\textrm{ht}(\lambda)=\sum_{i=1}^{\ell}c_i$: the height of $\lambda$, for $\la \in \mathbb{Z}\Phi$
\item $M^*=\Hom_k(M,k)$: the dual $G$-module of a given $G$-module $M$
\item $S^\bullet=S^\bullet(\fraku^*)$: Symmetric algebra over the dual space of $\fraku$ (the Lie algebra of $U$)
\item $\Lambda^\bullet=\Lambda^\bullet(\fraku^*)$: Exterior algebra over the dual space of $\fraku$
\end{itemize}

For a $G$-module $M$, we write $M^{(r)}$ for the module obtained by twisting the structure map for $M$ by $F_r$. Note that $G_r$ acts trivially on $M^{(r)}$.  Throughout the paper, tensor products will be taken over the field $k$ unless otherwise stated. 
Every $\la \in X(T)$ defines a one-dimensional $B$-module, which we also simply denote by $\la$.


For each $G$-module $M$, we define $\Ext^i_G(M,-)$ to be the $i$th right derived functor of $\Hom_G(M,-)$, see \cite[I.4.2]{Jan:2003} for details. Then the $n$th degree cohomology space of $G$ with coefficients in $M$ is defined to be
\[ \opH^n(G,M)=\Ext^n_G(k,M). \]
Recall also (cf. \cite[Corollary II.4.7]{Jan:2003}) that for any $n\ge 0$, we have
\[ \opH^n(G,M)\cong\opH^n(B,M). \]

In this paper, the Fibonacci sequence is defined by
\[F(1)=1,\quad F(2)=1, \quad  F(n)=F(n-1)+F(n-2), \]
for $n\ge 3$. 

We recall a result from \cite{BCNP}, which will play a key step in our main proof later.

\begin{proposition}\label{BCNP}
For any positive integers $m, n$, the number of sequences $\{x_i\}$ of non-negative integers such that 
\[
\begin{cases}
\displaystyle{\sum_{i=0}^\infty x_ip^i=m,}\\
\displaystyle{\sum_{i=0}^\infty x_i\le n}
\end{cases}
\]
is bounded by $F(n)$ if $p>3$, and bounded by $F(2n-1)$ for $p=2$ or $3$. Alternatively, one may say the number of partitions of $m$ into at most $n$ parts in the set $\{1, p, p^2, p^3, \ldots\}$ is bounded by $F(2n-1)$, which is independent of $m$ and $p$.
\end{proposition}

\begin{remark}
Note that it is rather surprising that upper bounds for the number of solutions to the system can be independent of $m$. This is crucial in our calculations as it helps establish bounds independent of weights. 
\end{remark}


 \section{Heights of Roots}
 
Before proving our main result in the next section, we first make a technical observation that is likely known, but for which we are unaware of a reference.  

\begin{lemma}\label{lem1}
For an irreducible root system of rank $\ell$, there are at most $\ell$ positive roots of the same height.
\end{lemma}

\begin{proof}
We conduct a case-by-case argument. Here we mainly use the epsilon basis for each classical Lie algebra type as in \cite{Hum}.\\

\noindent\underline{\bf Type $A_\ell$:} The set of positive roots consists of all $\epsilon_i-\epsilon_j=\alpha_i+\cdots+\alpha_{j-1}$ for $1\le i<j\le \ell+1$.  We have ht$(\epsilon_i - \epsilon_j) = j-i$. For each $h$ between $1$ and $\ell$, the number of height $h$ roots is the same as the number of pairs $(i,j)$ such that $j-i=h$ and $1\le i<j\le \ell+1$. Since $j=h+i$, it suffices to count the number of $i$s such that $h+i\le \ell+1$. Hence, there are exactly $\ell+1-h$ such roots, which is obviously at most $\ell$. \\

\noindent\underline{\bf Type $B_\ell$:} The set of positive roots consists of the following:
\begin{itemize}
\item[(1)] $\{\epsilon_i:1\le i\le \ell\}$ where each $\epsilon_i=\alpha_i+\alpha_{i+1}+\cdots+\alpha_\ell$ with ht$(\epsilon_i)=\ell-i+1$, and those heights are distinct.
\item[(2)] $\{\epsilon_i-\epsilon_j:1\le i<j\le \ell\}$ where each is of height $j-i$. Similar to type $A$, the number of these roots of height $h$ is $\ell-h$.
\item[(3)] $\{\epsilon_i+\epsilon_j:1\le i<j\le \ell\}$ where ht$(\epsilon_i+\epsilon_j)=2\ell+2-(i+j)$, ranging from 3 to $2\ell-1$. Observe that for $3\le h\le \ell$, the number of these roots of height $h$ is the same as the number of roots of height $2\ell+2-h$. In particular, this number can be counted by all the numbers $i$ such that $i<h-i\le \ell$, simplifying to $1\le i\le\lfloor\frac{h-1}{2}\rfloor$; therefore, the number of such roots is $\lfloor\frac{h-1}{2}\rfloor$. For $\ell+1\le h\le 2\ell-1$, the same argument can be applied to obtain the number of height $h$ roots is $\lfloor\frac{2\ell+2-h}{2}\rfloor$.  
\end{itemize}
In summary, the number of height 1 (involving roots in (1) and (2)) is $1+(\ell-1)=\ell$. Similarly, the number of height 2 (also involving roots in (1) and (2)) is $1+(\ell-2)=\ell-1<\ell$. For $3\le h\le \ell$, positive roots of height $h$ includes (1), (2), and (3), hence the number of such roots is
\[ 1+\ell-h+\left\lfloor\frac{h-1}{2}\right\rfloor < 1 + \ell - h + (h-1) = \ell. \]
For $h\ge\ell+1$, height $h$ roots only occur in (3) and recall the number of such roots is $\lfloor\frac{2\ell+2-h}{2}\rfloor=\lfloor\ell+1-\frac{h}{2}\rfloor<\ell$.\\

\noindent\underline{\bf Type $C_\ell$:} The set of positive roots consists of the following:
\begin{itemize}
\item[(1)] $\{\epsilon_i-\epsilon_j:1\le i<j\le \ell\}$ where each is of height $j-i$. Similar to type $A$, the number of these roots of height $h$ is $\ell-h$.
\item[(2)] $\{\epsilon_i+\epsilon_j:1\le i<j\le \ell\}$ where ht$(\epsilon_i+\epsilon_j)=2\ell+1-(i+j)$, ranging from 2 to $2\ell-2$. Observe that, for $2\le h\le \ell$, the number of these roots of height $h$ is the same as the number of these roots of height $2\ell-h$. In particular, this number can be counted by all the $i$s such that $i<h+1-i\le \ell$ simplifying to $1\le i\le\lfloor\frac{h}{2}\rfloor$; therefore, the number of such roots is $\lfloor\frac{h}{2}\rfloor$. For $\ell+1\le h\le 2\ell-2$, the same argument can be applied to obtain the number of height $h$ roots is $\lfloor\frac{2\ell-h}{2}\rfloor$. 
\item[(3)] $\{2\epsilon_i:1\le i\le \ell\}$ where each is of height $2\ell+1-2i$.  Note that ht$(2\epsilon_i)$ is always odd and each is distinct, with values ranging from 1 to $2\ell-1$. 
\end{itemize}
In summary, there are $\ell$ roots of height 1 as it involves roots in (1) and (3). For $2\le h\le \ell$, positive roots of height $h$ includes (1), (2), and (3), hence the number of such roots is
\[ \ell-h+\left\lfloor\frac{h}{2}\right\rfloor+1\le \ell - h + (h-1) + 1 = \ell. \]
For $h\ge\ell+1$, height $h$ roots only occur in (2) and (3) (if $h$ is odd). Hence, the number of roots in this case is at most $\lfloor\frac{2\ell-h}{2}\rfloor+1<\ell$.\\

\noindent\underline{\bf Type $D_\ell$:} The set of positive roots consists of the following:
\begin{itemize}
\item[(1)] $\{\epsilon_i-\epsilon_j:1\le i<j\le \ell\}$ where each is of height $j-i$. Similar to type $A$, the number of these roots of height $h$ is $\ell-h$.
\item[(2)] $\{\epsilon_i+\epsilon_j:1\le i<j<\ell\}$ where ht$(\epsilon_i+\epsilon_j)=2\ell-(i+j)$, ranging from 3 to $2\ell-3$. Observe that for $3\le h\le \ell$, the number of these roots of height $h$ is the same as the number of roots of height $2\ell-h$. In particular, this number can be counted by all the $i$s such that $i<h-i\le \ell$ simplifying to $1\le i\le\lfloor\frac{h-1}{2}\rfloor$; therefore, the number of such roots is $\lfloor\frac{h-1}{2}\rfloor$. For $\ell+1\le h\le 2\ell-3$, the same argument can be applied to obtain the number of height $h$ roots is $\lfloor\frac{2\ell-h-1}{2}\rfloor$. 
\item[(3)] $\{\epsilon_i+\epsilon_\ell:1\le i<\ell\}$ where each is of height $\ell-i$ and their height ranges from 1 to $\ell-1$. It's easy to see that there is only one root of each such height. 
\end{itemize}
In summary, there are $\ell$ roots of height 1 and $\ell-1$ roots of height 2 (involving roots in (1) and (2)). For $3\le h\le \ell-1$, positive roots of height $h$ includes (1), (2), and (3), hence the number of such roots is
\[ \ell-h+\left\lfloor\frac{h-1}{2}\right\rfloor + 1 < \ell - h + (h-1) + 1 = \ell. \]
For height $\ell$, we only count roots in (1) and (2). Hence, the number of roots in this case is \[1 + \left\lfloor\frac{\ell-1}{2}\right\rfloor <\ell. \]
Now for $h\ge\ell+1$, only roots in (2) will be counted for this height and there are $\lfloor\frac{2\ell-h-1}{2}\rfloor$, which is less than $\ell$.\\

The claim for root systems of exceptional type can be manually verified.
\end{proof}


\section{Linear Bounds for the Cohomology}\label{S:Main}

This section presents our main results. We begin with bounding the $B$-cohomology $\opH^n(B,\la)$ with coefficients in any one-dimensional $B$-module.

\begin{theorem}\label{T:p>2}
Let $\lambda\in X(T)$ and $n\ge 1$.  For $p > 3$,  
\[ \dim\opH^n(B,\lambda)\le F(n(h_G-1))\cdot\left(\frac{h_G}{2}\right)^n\cdot(2\ell)^{n(h_G-1)}. \]
For the case when $p=3$, we have
\[ \dim\opH^n(B,\lambda)\le F(2n(h_G-1)-1)\cdot \left(\frac{h_G}{2}\right)^n\cdot(2\ell)^{n(h_G-1)}. \]

\end{theorem}

\begin{proof}
By \cite[Theorem 7.1]{CPS}, for each $n\ge 1$ we have
\[ \opH^n(B,\lambda)\cong\varprojlim_r\opH^n(U_r,\lambda)^T. \]  
Hence, we will show that $\dim\opH^n(U_r,\lambda)^T$ is bounded by the number in the theorem for sufficiently large $r$.

Applying $T$-invariants to the spectral sequence in \cite[I.9.14]{Jan:2003}, we have the spectral sequence of $B$-modules:
\begin{equation}\label{U_rspectral}
E^{i,j}_1=\bigoplus \left(\lambda\otimes S^{a_1(1)}\otimes\cdots\otimes S^{a_r(r)}\otimes\Lambda^{b_1}\otimes\Lambda^{b_2(1)}\otimes\cdots\otimes\Lambda^{b_r(r-1)}\right)^{T}\Rightarrow\left(\opH^{i+j}(U_r,\lambda)\right)^{T}
\end{equation}
where the direct sum is taken over all $a_i$s and $b_j$s satisfying 
\[
\begin{cases}
 i+j  & =  2(a_1+\cdots+a_r)+b_1+\cdots+b_r, \\
 i & =  \sum_{n=1}^{r}(a_np^n+b_np^{n-1}).
\end{cases}
\]
It's easy to see that $\opH^n(B,\lambda)=0$ unless $\lambda$ is a non-negative sum of negative roots. Now assume $\lambda$ is in $\mathbb{N}\Phi^-$. It suffices to show that $\dim\bigoplus_{i+j=n}E^{i,j}_1$ has the stated bound.

Recall that $S^c$ (resp. $\Lambda^c$) is the $c$th symmetric (resp. exterior) power of $\fraku^*$. Hence, each weight of $S^{c}$ or $\Lambda^{c}$ is a sum of $c$ positive roots (necessarily distinct in the latter case). For any $r$, $\dim\opH^n(U_r,\lambda)^T$ is bounded by the number of root solutions to the system
\begin{align*}
\begin{cases}
\displaystyle{\beta(b_1)+\beta(a_1+b_2)p+\cdots+\beta(a_{r-1}+b_{r})p^{r-1} + \beta(a_r)p^r =-\lambda,}\\
\displaystyle{\sum_{i=0}^r(a_i + b_{i+1}) = \sum_{i=1}^r(a_i+b_i) \le \sum_{i = 1}^r(2a_i + b_i) = n},
\end{cases}
\end{align*}
where, for each $0\le i\le r$, $\beta(a_i+b_{i+1})$ is a sum of $(a_i+b_{i+1})$ positive roots with $a_0=b_{r+1}=0$. In other words, we are searching for the number of ways to choose at most $n$ positive roots to satisfy the first equation. Since $\lambda$ is fixed, the number of solutions stabilizes as $r$ increases.  And so it suffices to find a bound on this number for sufficiently large $r$. 

To simplify the estimation (at the expense of over counting), we estimate the number of root solutions such that in the first equation the heights of each side agree.  We break this estimation into the steps below.

\bigskip

(i) Consider the first equation and the height of the weight $-\lambda$. Set $h_i := \textrm{ht}(\beta(a_i+b_{i+1}))$.  Then we have
$$
h_0+h_1p+\cdots+h_rp^{r}=\textrm{ht}(-\lambda).
$$
In total, there are at most $n$ roots, and any positive root has height at most $h_G-1$. Therefore, the sum of the $h_i$ is at most $n(h_G - 1)$.  It follows that we are considering the number of integer solutions to
\begin{align*}
\begin{cases}
\displaystyle{h_0+h_1p+\cdots+h_rp^{r}=\textrm{ht}(-\lambda),}\\
\displaystyle{\sum_{i=0}^rh_i\le n(h_G-1).}
\end{cases}
\end{align*}
By Prop. \ref{BCNP}, this number is bounded by $F(n(h_G-1))$ if $p>3$ and by $F(2n(h_G-1)-1)$ if $p=3$.\\

\bigskip

(ii) Take a family of $\{h_i : 0 \leq i \leq r\}$ and consider a given (non-zero) $h_i$.   We first consider all 2-part decompositions of $h_i$ as $h_i = h_{i,1} + h_{i,2}$ for $0 \leq h_{i,j} \leq h_i$. This corresponds to separating the $\beta(a_i + b_{i+1})$ roots into a symmetric component $\beta(a_i)$ and exterior component $\beta(b_{i+1})$ with $h_{i,1} = \text{ht}(\beta(a_i))$ and $h_{i,2} = \text{ht}(\beta(b_{i+1}))$. Clearly there are $h_i + 1$ such decompositions and over all $i$ we have $\prod_{i=0}^r(h_i+1)=\prod_{h_i \neq 0}(h_i+1)$. We can bound this product as follows. Note that the number of non-zero $h_i$s can't exceed $n$, so we can reindex $\{h_i~|~0 \leq i \leq r\}$ as $\{\tilde{h}_j~|~ 1 \leq j \leq r + 1\}$, where $\tilde{h}_j = 0$ if $j > n$. Then we apply the arithmetic-geometric mean inequality to have 

\begin{align*}
\prod_{h_i \neq 0}(h_i+1)=\prod_{j=1}^n(\tilde{h}_j+1)&\le \left(\frac{\sum_{j=1}^n(\tilde{h}_{j}+1)}{n}\right)^n\\
&=\left(\frac{\sum_{j=1}^n\tilde{h}_{j}+n}{n}\right)^n\\
&\le \left(\frac{nh_G}{n}\right)^n = h_G^n.
\end{align*}
\bigskip

(iii)  For a given decomposition $h_i = h_{i,1} + h_{i,2}$, we now need to consider in how many ways roots may be chosen to give the specified heights.   For a given $h_{i,j}$, we split this into two steps: (a) identifying the number of ``root frameworks'' allowed and (b) counting the number of roots that can fit into that framework. For example, suppose that $h_{i,1} = 3$.   Then $\beta(a_i)$ is a sum of positive roots with total height of 3.   This could be obtained by taking a single root of height 3, a root of height two and another of height 1, or three roots each of height 1.  We refer to each of these options as a ``root framework.''  Observe that a root framework for $h_{i,j}$ is essentially just a partition of $h_{i,j}$.  E.g., the partitions of 3: (3), (2,1), and (1,1,1).   For a positive integer $x$, let $P(x)$ denote the number of partitions of $x$ and recall that $P(x) \leq 2^{x-1}$.  Set $P(0) = 1$.  Therefore, for a given pair $h_{i,1}$ and $h_{i,2}$, the total number of root frameworks is 
$$
P(h_{i,1})\cdot P(h_{i,2}) \leq 2^{h_{i,1} + h_{i,2} - 1} = 2^{h_i - 1},
$$
noting that in many cases one could say ``$-2$'' rather than ``$-1$''. Overall, reindexing as we did in the end of (ii), we obtain
\[ \prod_{i=0}^rP(h_{i,1})\cdot P(h_{i,2})=\prod_{h_i \neq 0}2^{h_i-1}=\prod_{j=1}^n2^{\tilde{h}_{j}-1}=2^{\sum_{j=1}^n\tilde{h}_{j}-n}\le \frac{2^{n(h_G-1)}}{2^n}. \]

\bigskip

(iv) For each root framework of height $h_i=h_{i,1} + h_{i,2}$ in (iii), we now need to consider the number of ways we can choose roots for that framework.  E.g., in how many ways can we choose a root of height 3 or in how many ways can we choose one of height two and one of height one or three of height one.   As observed in Lemma \ref{lem1}, there are at most $\ell$ ways to choose a root of a given height.   In our running example, this would give us at most $\ell$ ways to choose a root of height 3, $\ell^2$ ways to choose a pair of roots (one of height 2 and one of height 1), and $\ell^3$ ways to choose 3 roots of height one.  In general, for a given $h_{i,j}$, the worst case is always $\ell^{h_{i,j}}$ options.  Thus, for a given decomposition $h_i = h_{i,1} + h_{i,2}$ and specific choice of root framework from (ii), there are at most $\ell^{h_{i,1}}\ell^{h_{i,2}}=\ell^{h_i}$ ways to fill in that root framework. Overall, for each solution $\{h_i : 0 \leq i \leq r\}$ in step (i), we have 
\[ \prod_{i=0}^r\ell^{h_i}=\ell^{\sum_{i=0}^rh_i}\le\ell^{n(h_G-1)}. \] 
Note that the exterior algebra adds another layer of over estimation here, since our estimate allows one to repeat roots of a specified height, but that would not be allowed in the exterior component.

\bigskip
Finally, combining all the steps gives the bound for the cohomology dimension as stated.

\end{proof}

For the case when $p=2$, in place of \eqref{U_rspectral}, we would use the spectral sequence
\begin{equation*}
E^{i,j}_1=\bigoplus\left( \lambda\otimes S^{a_1}\otimes\cdots\otimes S^{a_r(r-1)}\right)^{T}\Rightarrow\left(\opH^{i+j}(U_r,\lambda)\right)^{T},
\end{equation*}
where the direct sum is taken over all $a_i$s satisfying 
\begin{equation*}
\left\{ \begin{array}{rcl} i+j  & = & a_1+\cdots+a_r, \\
 i & = & \sum_{n=1}^{r}a_np^{n-1}. \end{array}\right.
\end{equation*}
Then the system in our proof would become
\begin{align*}
\begin{cases}
\displaystyle{\beta(a_1)+\beta(a_2)p+\cdots+\beta(a_{r})p^{r-1} =-\lambda,}\\
\displaystyle{\sum_{i=1}^ra_i= n}.
\end{cases}
\end{align*}
Now the same argument as above may be applied to this case, but it is slightly simpler as we do not need to consider the step (ii).  Consequently, we obtain the following.

\begin{theorem}\label{T:p=2}
Assume $p = 2$. For any $\lambda\in X(T)$, we have for all $n\ge 1$ 
\[ \dim\opH^n(B,\lambda)\le F(2n(h_G-1)-1)\cdot 2^{n(h_G - 2)}\ell^{n(h_G-1)}. \]
\end{theorem}

\begin{remark}
As noted in the proof, our upper bounds in Theorems \ref{T:p>2} and \ref{T:p=2} are very far from tight, as we significantly over counted possible solutions formed in steps (iii) and (iv) with our exponential estimates. We also ignored various items:  whether a root framework in step (iii) actually is possible, whether the number of parts $h_{i,1}+h_{i,2}$ in each root framework from step (ii) must add up to $n$, and whether all potential $\beta(a_i+b_{i+1})$s satisfy the equation
\[ \beta(b_1)+\beta(a_1+b_2)p+\cdots+\beta(a_{r-1}+b_{r})p^{r-1} + \beta(a_r)p^r =-\lambda.\]
Also, the fact that we are only looking at the $E_1$-page of a spectral sequence suggests further over counting. 
\end{remark}

For every $G$-module $M$, recall the fact that $\opH^n(G,M) \cong \opH^n(B,M)$. As a $B$-module, $M$ has precisely $\dim(M)$ weights or one-dimensional composition factors.  Inductively, using a long exact sequence in cohomology, from the above theorems, we finally establish a linear bound for $G$-cohomology in the following.

\begin{corollary}\label{C:main}
For any finite-dimensional rational $G$-module $M$, we have for all $n\ge 1$,
\[  \dim\opH^n(G, M) \le F(2n(h_G-1)-1)\cdot \left(\frac{h_G}{2}\right)^n\cdot(2\ell)^{n(h_G-1)}\cdot\dim(M). \]
\end{corollary}

\begin{remark} For $p > 3$, one may replace $F(2n(h_G-1) - 1)$ with $F(n(h_G-1))$, and, for $p = 2$, we may replace the term $\left(\frac{h_G}{2}\right)^n$ with $\frac{1}{2^n}$.  In general, additional over estimating has occurred here as well, since we have $H^n(B,\la) = 0$ for $\la \notin \mathbb{N}\Phi^{-}$.
\end{remark}

\begin{remark} For $n = 1$ and $p > 3$, our bound is 
\begin{equation}\label{E:n=1}
\dim\opH^1(G,M) \le F(h_g-1)\cdot \frac{h_G}{2}\cdot (2\ell)^{(h_G-1)}\cdot\dim{M}.
\end{equation}
As mentioned in the introduction a bound for irreducible modules was given in \cite{PaSt} that depends somewhat on $p$. For $p \geq h_G$, their bound could be naively translated into the bound
\begin{equation}\label{E:n=1-2}
\dim\opH^1(G,M) \le \frac{z^z - 1}{z - 1}\cdot\dim(M),
\end{equation}
where $z = \left\lfloor\frac{h_G^3}{6}\right\rfloor$.  For an elementary comparison, consider a group of type $B_3$ or $C_3$. Then $\ell = 3$, $h_G = 6$, and $z = 36$.   Then the bound from (\ref{E:n=1}) is 
$$
F(6 - 1)\cdot\frac{6}{2}\cdot(2\cdot 3)^{6 - 1}\dim(M) = 15\cdot 6^5\cdot \dim(M), 
$$
whereas (\ref{E:n=1-2}) becomes
$$
\frac{36^{36} - 1}{35}\cdot\dim(M),
$$
which is substantially larger. For higher rank groups, their bound would grow faster as it is of the exponential form $(\frac{h_G^3}{6})^{\frac{h_G^3}{6}}$ while ours is of the form $(2\ell)^{h_G}$. 
\end{remark}

\providecommand{\bysame}{\leavevmode\hbox to3em{\hrulefill}\thinspace}

\end{document}